\documentclass{scrartcl}

\usepackage[utf8]{luainputenc}
\usepackage[USenglish]{babel}
\usepackage{csquotes}

\usepackage[a4paper,top=27mm,bottom=20mm,inner=25mm,outer=20mm]{geometry}

\usepackage[%
  backend=bibtex,bibencoding=ascii,
  style=numeric-comp,
  giveninits=true, uniquename=init, 
  natbib=true,
  url=true,
  doi=true,
  isbn=false,
  backref=false,
  maxnames=99,
  ]{biblatex}
\usepackage{amsmath}
\allowdisplaybreaks
\numberwithin{equation}{section}
\usepackage{amssymb}
\usepackage{commath}
\usepackage{mathtools}
\usepackage{bbm}
\usepackage{nicefrac}
\usepackage{subdepth}

\usepackage{siunitx}
\usepackage{amsthm}
\usepackage{thmtools}
\usepackage{etoolbox}
\makeatletter
\patchcmd{\thmt@setheadstyle}
 {\bgroup\thmt@space}
 {\thmt@space}
 {}{}
\patchcmd{\thmt@setheadstyle}
 {\egroup\fi}
 {\fi}
 {}{}
\makeatother
\declaretheoremstyle[
  bodyfont=\normalfont\itshape,
  headformat=\NAME\ \NUMBER\NOTE,
]{myplain}
\declaretheoremstyle[
  headformat=\NAME\ \NUMBER\NOTE,
]{mydefinition}
\newcommand{\envqed}{{\lower-0.3ex\hbox{$\triangleleft$}}}
\declaretheorem[style=myplain,numberwithin=section]{theorem}
\declaretheorem[style=myplain,numberlike=theorem]{lemma}

\declaretheorem[style=mydefinition,numberlike=theorem,qed=\envqed]{definition}

\usepackage[plainpages=false,pdfpagelabels,hidelinks,unicode]{hyperref}

\usepackage{color}
\usepackage{graphicx}
\usepackage[small]{caption}
\usepackage{subcaption}

\begingroup\expandafter\expandafter\expandafter\endgroup
\expandafter\ifx\csname pdfsuppresswarningpagegroup\endcsname\relax
\else
\fi

\usepackage{booktabs}
\usepackage{rotating}
\usepackage{multirow}

\usepackage{enumitem}

\usepackage{ifluatex}
\ifluatex
  \usepackage[no-math]{fontspec}
\else
  \usepackage[T1]{fontenc}
\fi
\usepackage{newpxtext,newpxmath}

\usepackage{xparse}

\let\epsilon\varepsilon
\let\phi\varphi
\let\rho\varrho

\usepackage{xspace}

\newcommand{\const}{\mathrm{const}}

\renewcommand{\vec}[1]{\pmb{#1}}

\NewDocumentCommand{\opD}{m+g}{%
  \IfNoValueTF{#2}
    {D_{#1}}
    {D_{#1,#2}}%
}
\NewDocumentCommand{\opDsplit}{m+g}{%
  \IfNoValueTF{#2}
    {\widetilde{D}_{#1}}
    {\widetilde{D}_{#1,#2}}%
}
\NewDocumentCommand{\opM}{g}{%
  \IfNoValueTF{#1}
    {M}
    {M_{#1}}%
}
\NewDocumentCommand{\opQ}{g}{%
  \IfNoValueTF{#1}
    {Q}
    {Q_{#1}}%
}
\NewDocumentCommand{\opI}{g}{%
  \IfNoValueTF{#1}
    {I}
    {I_{#1}}%
}
\NewDocumentCommand{\opV}{g}{%
  \IfNoValueTF{#1}
    {V}
    {V_{#1}}%
}
\NewDocumentCommand{\opB}{g}{%
  \IfNoValueTF{#1}
    {B}
    {B_{#1}}%
}
\NewDocumentCommand{\opR}{g}{%
  \IfNoValueTF{#1}
    {R}
    {R_{#1}}%
}
\NewDocumentCommand{\opN}{m+g}{%
  \IfNoValueTF{#2}
    {N_{#1}}
    {N_{#1,#2}}%
}
\NewDocumentCommand{\x}{g}{%
  \IfNoValueTF{#1}
    {\vec{x}}
    {\vec{x}_{#1}}%
}

\NewDocumentCommand{\fnum}{g}{%
  \IfNoValueTF{#1}
    {f^{\mathrm{num}}}
    {f^{\mathrm{num,#1}}}%
}
\NewDocumentCommand{\vecfnum}{g}{%
  \IfNoValueTF{#1}
    {\vec{f}^{\mathrm{num}}}
    {\vec{f}^{\mathrm{num,#1}}}%
}
\NewDocumentCommand{\vecfcorr}{g}{%
  \IfNoValueTF{#1}
    {\vec{f}^{\mathrm{corr}}}
    {\vec{f}^{\mathrm{corr,#1}}}%
}
\NewDocumentCommand{\fvol}{g}{%
  \IfNoValueTF{#1}
    {f^{\smash{\mathrm{vol}}}}
    {f^{\smash{\mathrm{vol,#1}}}}%
}

\newcommand{\mean}[1]{{\{\mkern-6mu\{}#1{\}\mkern-6mu\}}}

\newcommand{\jump}[1]{{[\mkern-3mu[}#1{]\mkern-3mu]}}

\newcommand{\orcid}[1]{ORCID:~\href{https://orcid.org/#1}{#1}}
\usepackage{authblk}

\newenvironment{keywords}{\par\textbf{Key words.}}{\par}

\title{A Note on Numerical Fluxes Conserving a Member of Harten's One-Parameter Family of Entropies for the Compressible Euler Equations}

\author[1]{Hendrik~Ranocha\thanks{\orcid{0000-0002-3456-2277}}}
\affil[1]{Applied Mathematics, University of Hamburg, Germany}

\date{March 8, 2022}

\makeatletter
\hypersetup{pdfauthor={Hendrik Ranocha}}
\hypersetup{pdftitle={A Note on Numerical Fluxes Conserving a Member of Harten's One-Parameter Family of Entropies for the Compressible Euler Equations}}
\makeatother

\begin{document}

\maketitle

\begin{abstract}
\noindent
  Entropy-conserving numerical fluxes are a cornerstone of modern high-order
entropy-dissipative discretizations of conservation laws. In addition to entropy
conservation, other structural properties mimicking the continuous level such as
pressure equilibrium and kinetic energy preservation are important. This note
proves that there are no numerical fluxes conserving (one of) Harten's entropies
for the compressible Euler equations that also preserve pressure equilibria and
have a density flux independent of the pressure. This is in contrast to fluxes
based on the physical entropy, where even kinetic energy preservation can be
achieved in addition.

\end{abstract}


\section{Introduction}
\label{sec:introduction}

Ever since the seminal work of Tadmor \cite{tadmor1987numerical}, researchers
have been interested in entropy-dissipative numerical methods for conservation laws
and related models. Usually, these methods have improved robustness properties,
even for underresolved simulations. Nowadays, several means have been explored
to ensure entropy stability. One of the most popular and successful approaches
is based on entropy-conservative (EC) numerical fluxes. These can be used to
construct high-order central-type methods using flux differencing \cite{fisher2013high}
to which appropriate dissipation can be added. A recent alternative is the general
algebraic approach of \cite{abgrall2018general,abgrall2022reinterpretation}.

There are several numerical fluxes for the compressible Euler equations
conserving the physical (logarithmic) entropy
\cite{ismail2009affordable,chandrashekar2013kinetic,ranocha2018comparison}.
Harten \cite{harten1983symmetric} studied another family of entropies for the
compressible Euler equations. Although these entropies do not symmetrize the heat
flux terms in the compressible Navier-Stokes equations \cite{hughes1986new}, they are still of
interest, for example to construct entropy splitting methods or related
EC fluxes \cite{sjogreen2019entropy}.

Entropy conservation alone is often insufficient to construct good numerical
methods. Of course, dissipation and related issues such as shock capturing and
positivity of the density and internal energy are also important but are not the
focus of this contribution. Instead, preservation of the kinetic energy
\cite{jameson2008formulation,ranocha2018thesis,ranocha2020entropy}
and pressure equilibria \cite{ranocha2021preventing,shima2021preventing}
is considered. Moreover, the numerical density flux should not depend on the
pressure, in accordance with physical expectations. This is discussed further in
\cite{derigs2017novel,ranocha2018comparison}, where positivity failure could be
identified for certain setups with large pressure jumps and constant densities
and velocities, even in the presence of strong (numerical) dissipation.

The main contribution of this note is to prove that there are no numerical fluxes
conserving Harten's entropies for the compressible Euler equations that also
preserve pressure equilibria and have a density flux independent of the pressure
(Section~\ref{sec:main}). Further discussion of this result is presented in
Section~\ref{sec:discussion}.

\section{Main result}
\label{sec:main}

It suffices to concentrate on the 1D compressible Euler equations
\begin{equation}
\label{eq:euler}
  \partial_t \vec{u} + \partial_x \vec{f}(\vec{u}) = \vec{0},
  \qquad
  \vec{u} = (\rho, \rho v, \rho e)^T,
  \quad
  \vec{f}(\vec{u}) = (\rho v, \rho v^2 + p, (\rho e + p) v)^T,
\end{equation}
where $\rho$ is the density, $v$ the velocity, $\rho e$ the total energy, and
the pressure $p = (\gamma - 1) (\rho e - \rho v^2 / 2)$ is given by the ideal
gas law with ratio of specific heats $\gamma > 1$.
Harten \cite{harten1983symmetric} considered entropies of the form
$U(\vec{u}) = - \rho h(s)$, where $s =  \log(p / \rho^\gamma)$ and $h$ is a smooth
function satisfying $h''(s) / h'(s) < 1 / \gamma$. A special one-parameter
family of these entropies is given by
\begin{equation}
\label{eq:U}
  U(\vec{u}) =
  -\frac{\gamma + \alpha}{\gamma - 1} \rho (p / \rho^\gamma)^{1 / (\alpha + \gamma)},
  \qquad
  \alpha > 0 \text{ or } \alpha < -\gamma,
\end{equation}
where the restriction of the parameter $\alpha$ ensures convexity of $U$
\cite{sjogreen2019entropy}. The members of this one-parameter family \eqref{eq:U}
of entropy functions for the compressible Euler equations are often referred to
as Harten's entropies in the literature, e.g., in \cite{sjogreen2019entropy}.
The associated entropy variables $\vec{w} = U'(\vec{u})$ and the flux potential $\psi$
\cite{tadmor1987numerical} are given by
\begin{equation}
\label{eq:w-psi}
  \vec{w} = \frac{\rho}{p} (p / \rho^\gamma)^{1 / (\alpha + \gamma)} \left(
    -\frac{\alpha}{\gamma - 1} \frac{p}{\rho} - \frac{1}{2} v^2, v, -1
  \right),
  \qquad
  \psi = \rho (p / \rho^\gamma)^{1 / (\alpha + \gamma)} v.
\end{equation}
This note focuses on two-point numerical fluxes for the compressible Euler
equations. Such a two-point numerical flux $\vecfnum$ is characterized as follows.
\begin{definition}[Entropy conservation \cite{tadmor1987numerical}]
\label{def:ec}
  The numerical flux $\vecfnum$ is EC if $\jump{\vec{w}} \cdot \vecfnum - \jump{\psi} = 0$,
  where $\jump{a} = a_+ - a_-$ is the common jump operator.
\end{definition}
\begin{definition}[Pressure equilibrium preservation \cite{ranocha2021preventing}]
\label{def:pep}
  A numerical flux $\vecfnum = (\vecfnum_\rho, \vecfnum_{\rho v}, \vecfnum_{\rho e})$ is
  pressure equilibrium preserving (PEP) if
  $\vecfnum_{\rho v} = v \vecfnum_{\rho} + \const_1(p, v)$
  and
  $\vecfnum_{\rho e} = \frac{1}{2} v^2 \vecfnum_\rho + \const_2(p, v)$
  whenever the velocity $v$ and the pressure $p$ are constant.
  Here, $\const_i(p, v)$, $i \in \{ 1, 2\}$, denote some generic constants
  depending only on the constants $v$ and $p$.
\end{definition}
As discussed in \cite{shima2021preventing,ranocha2021preventing,gassner2022stability},
pressure equilibria are important setups where the velocity and the pressure are
constant. In this case, the compressible Euler equations are reduced to linear
advection equations. Pressure equilibrium preserving schemes keep this property
at the discrete level \cite{ranocha2021preventing}.

The main result of this note is
\begin{theorem}
\label{thm:main}
  There is no two-point numerical flux for the compressible Euler equations \eqref{eq:euler}
  that is entropy-conserving in the sense of Tadmor (Def.~\ref{def:ec})
  for a member of Harten's one-parameter family of entropies \eqref{eq:U},
  pressure equilibrium preserving (Def.~\ref{def:pep}), and has a density flux
  that does not depend on the pressure.
\end{theorem}

The proof of Theorem~\ref{thm:main} is divided into the following steps.
\begin{lemma}
\label{lem:p-v-const}
  For $p \equiv \const$ and $v \equiv \const$, a numerical flux that is PEP and
  EC for a member of Harten's one-parameter family of entropies \eqref{eq:U}
  has a density flux of the form
  \begin{equation}
  \label{eq:p-v-const}
    \vecfnum_{\rho}
    =
    -\frac{\gamma}{\alpha}
    \frac{ \jump{\rho^{\alpha / (\alpha + \gamma)}} }
         { \jump{\rho^{-\gamma / (\alpha + \gamma)}} } v.
  \end{equation}
\end{lemma}
\begin{proof}
  For constant pressure and velocity,
  \begin{equation}
  \begin{aligned}
    \jump{\vec{w}} \cdot \vecfnum - \jump{\psi}
    &=
    \left(
      -\frac{\alpha}{\gamma - 1} p^{1 / (\alpha + \gamma)} \jump{\rho^{-\gamma / (\alpha + \gamma)}}
      - \frac{1}{2} p^{(1 - \alpha - \gamma) / (\alpha + \gamma)} v^2 \jump{\rho^{\alpha / (\alpha + \gamma)}}
    \right) \vecfnum_{\rho}
    \\&\quad
    + p^{(1 - \alpha - \gamma) / (\alpha + \gamma)} v \jump{\rho^{\alpha / (\alpha + \gamma)}}
    \vecfnum_{\rho v}
    - p^{(1 - \alpha - \gamma) / (\alpha + \gamma)} \jump{\rho^{\alpha / (\alpha + \gamma)}}
    \vecfnum_{\rho e}
    \\&\quad
    - p^{1 / (\alpha + \gamma)} v \jump{\rho^{\alpha / (\alpha + \gamma)}}.
  \end{aligned}
  \end{equation}
  A PEP flux is of the form $\vecfnum_{\rho v} = v \vecfnum_{\rho} + p$,
  $\vecfnum_{\rho e} = \frac{1}{2} v^2 \vecfnum_\rho + p v \gamma / (\gamma - 1)$.
  Thus,
  \begin{equation}
  \begin{aligned}
    \jump{\vec{w}} \cdot \vecfnum - \jump{\psi}
    &=
    \left(
      -\frac{\alpha}{\gamma - 1} p^{1 / (\alpha + \gamma)} \jump{\rho^{-\gamma / (\alpha + \gamma)}}
      - \frac{1}{2} p^{(1 - \alpha - \gamma) / (\alpha + \gamma)} v^2 \jump{\rho^{\alpha / (\alpha + \gamma)}}
    \right) \vecfnum_{\rho}
    \\&\quad
    + p^{(1 - \alpha - \gamma) / (\alpha + \gamma)} v^2 \jump{\rho^{\alpha / (\alpha + \gamma)}}
    \vecfnum_{\rho}
    + p^{1 / (\alpha + \gamma)} v \jump{\rho^{\alpha / (\alpha + \gamma)}}
    \\&\quad
    - \frac{1}{2} p^{(1 - \alpha - \gamma) / (\alpha + \gamma)} v^2 \jump{\rho^{\alpha / (\alpha + \gamma)}}
    \vecfnum_{\rho}
    - \frac{\gamma}{\gamma - 1} p^{1 / (\alpha + \gamma)} v \jump{\rho^{\alpha / (\alpha + \gamma)}}
    \\&\quad
    - p^{1 / (\alpha + \gamma)} v \jump{\rho^{\alpha / (\alpha + \gamma)}}
    \\
    &=
    -\frac{\alpha}{\gamma - 1} p^{1 / (\alpha + \gamma)} \jump{\rho^{-\gamma / (\alpha + \gamma)}}
    \vecfnum_{\rho}
    - \frac{\gamma}{\gamma - 1} p^{1 / (\alpha + \gamma)} v \jump{\rho^{\alpha / (\alpha + \gamma)}}.
  \end{aligned}
  \end{equation}
  Hence, the EC condition is equivalent to \eqref{eq:p-v-const}.
\end{proof}

\begin{lemma}
\label{lem:p-special}
  For $v \equiv \const$ and $p_- = p_+ (\rho_- / \rho_+)^{\alpha / (\alpha + \gamma - 1)}$,
  a numerical flux that is EC for a member of Harten's one-parameter family
  of entropies \eqref{eq:U} has a density flux of the form
  \begin{equation}
  \label{eq:p-special}
    \vecfnum_{\rho}
    =
    -\frac{\gamma - 1}{\alpha}
    \frac{ \jump{\rho^{\alpha / (\alpha + \gamma - 1)}} }
         { \jump{\rho^{(1 - \gamma) / (\alpha + \gamma - 1)}} } v.
  \end{equation}
\end{lemma}
\begin{proof}
  For this special choice of the pressure,
  \begin{equation}
  \begin{split}
    \frac{\rho_-}{p_-} (p_- / \rho_-^\gamma)^{1 / (\alpha + \gamma)}
    =
    \rho_-^{\alpha / (\alpha + \gamma)} p_-^{(1 - \alpha - \gamma) / (\alpha + \gamma)}
    =
    \rho_-^{\alpha / (\alpha + \gamma)} (\rho_- / \rho_+)^{-\alpha / (\alpha + \gamma)} p_+^{(1 - \alpha - \gamma) / (\alpha + \gamma)}
    \\
    =
    \rho_+^{\alpha / (\alpha + \gamma)} p_+^{(1 - \alpha - \gamma) / (\alpha + \gamma)}
    =
    \frac{\rho_+}{p_+} (p_+ / \rho_+^\gamma)^{1 / (\alpha + \gamma)},
  \end{split}
  \end{equation}
  i.e., $\jump{(\rho/p) (p / \rho^\gamma)^{1 / (\alpha + \gamma)}} = 0$. Thus,
  for $v \equiv \const$, \eqref{eq:w-psi} yields
  \begin{equation}
  \begin{aligned}
    \jump{\vec{w}} \cdot \vecfnum - \jump{\psi}
    &=
    - \frac{\alpha}{\gamma - 1} \jump{(p / \rho^\gamma)^{1 / (\alpha + \gamma)}}
    \vecfnum_{\rho}
    - \jump{\rho (p / \rho^\gamma)^{1 / (\alpha + \gamma)}} v.
  \end{aligned}
  \end{equation}
  Consequently, an EC flux must be of the form
  \begin{equation}
    \vecfnum_{\rho}
    =
    - \frac{\gamma - 1}{\alpha}
    \frac{ \jump{\rho (p / \rho^\gamma)^{1 / (\alpha + \gamma)}} }
         { \jump{(p / \rho^\gamma)^{1 / (\alpha + \gamma)}} } v.
  \end{equation}
  Inserting the pressure ratio, the fraction of the jump terms can be written as
  \begin{equation}
  \begin{split}
    \frac{   \rho_+^{\alpha / (\alpha + \gamma)} p_+^{1 / (\alpha + \gamma)}
           - \rho_-^{\alpha / (\alpha + \gamma)} p_-^{1 / (\alpha + \gamma)} }
         {   \rho_+^{-\gamma / (\alpha + \gamma)} p_+^{1 / (\alpha + \gamma)}
           - \rho_-^{-\gamma / (\alpha + \gamma)} p_-^{1 / (\alpha + \gamma)} }
    =
    \frac{   \rho_+^{\alpha / (\alpha + \gamma)}
           - \rho_-^{\alpha / (\alpha + \gamma)} (\rho_- / \rho_+)^{\alpha / ((\alpha + \gamma - 1) (\alpha + \gamma))} }
         {   \rho_+^{-\gamma / (\alpha + \gamma)}
           - \rho_-^{-\gamma / (\alpha + \gamma)} (\rho_- / \rho_+)^{\alpha / ((\alpha + \gamma - 1) (\alpha + \gamma))} }
    \\
    =
    \frac{   \rho_+^{\alpha / (\alpha + \gamma - 1)}
           - \rho_-^{\alpha / (\alpha + \gamma - 1)} }
         {   \rho_+^{-\gamma / (\alpha + \gamma - 1)}
           - \rho_-^{-\gamma / (\alpha + \gamma - 1)} },
  \end{split}
  \end{equation}
  proving \eqref{eq:p-special}.
\end{proof}

\begin{proof}[Proof of Theorem~\ref{thm:main}]
  For given $\alpha, \gamma$, and $\rho_\pm$, such a flux needs to satisfy both
  \eqref{eq:p-v-const} and \eqref{eq:p-special}, i.e.,
  \begin{equation}
    \gamma
    \frac{ \jump{\rho^{\alpha / (\alpha + \gamma)}} }
         { \jump{\rho^{-\gamma / (\alpha + \gamma)}} }
    =
    (\gamma - 1)
    \frac{ \jump{\rho^{\alpha / (\alpha + \gamma - 1)}} }
         { \jump{\rho^{(1 - \gamma) / (\alpha + \gamma - 1)}} }.
  \end{equation}
  For fixed $\alpha$ and $\gamma$, it is easy to find $\rho_- \ne \rho_+$ such
  that this equation is not satisfied. Hence, a numerical flux with all properties
  listed in Theorem~\ref{thm:main} cannot exist.
\end{proof}

\section{Discussion}
\label{sec:discussion}

There are no numerical fluxes conserving a member of Harten's one-parameter
family of entropies \eqref{eq:U} for the
compressible Euler equations that also preserve pressure equilibria and have a
density flux independent of the pressure (Theorem~\ref{thm:main}).
This result is in contrast to fluxes conserving the physical (logarithmic) entropy,
where kinetic energy preservation can be achieved in addition to all of these
properties \cite{ranocha2018thesis,ranocha2020entropy}, resulting in an essentially
unique numerical flux \cite{ranocha2021preventing}.

Following the uniqueness proof of such a numerical flux based on the logarithmic
entropy presented in \cite{ranocha2021preventing}, it is tempting to choose
a density flux based on a simplified setting, e.g., for constant velocity and
pressure as in Lemma~\ref{lem:p-v-const}. Then, kinetic energy preservation
determines the momentum flux accordingly \cite{ranocha2018thesis,ranocha2020entropy}
and one might expect to be able to use the EC criterion to find an energy flux.
However, this does not work in general since the energy part of the numerical flux
can be orthogonal to the jump of the entropy variables as in Lemma~\ref{lem:p-special}.
In this case, such an approach will often lead to a blow-up of the
derived component for the total energy flux.
Hence, one cannot impose a form of the density flux in general.

\begin{figure}[!ht]
\centering
  \begin{subfigure}{0.33\textwidth}
  \centering
    \includegraphics[width=\linewidth]{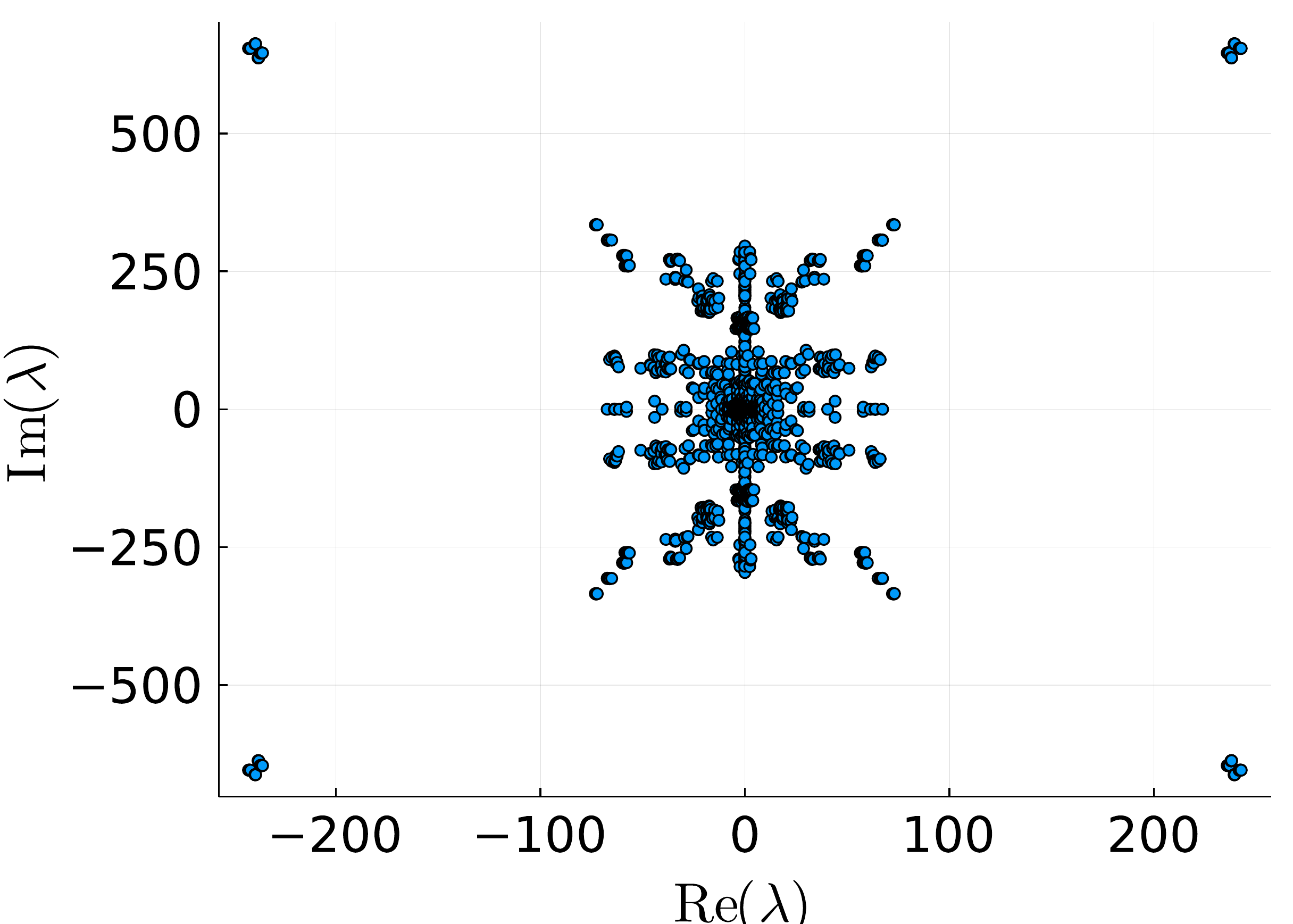}
    \caption{$\alpha = -1.8$, $\beta = 1$.}
  \end{subfigure}%
  \hspace*{\fill}
  \begin{subfigure}{0.33\textwidth}
  \centering
    \includegraphics[width=\linewidth]{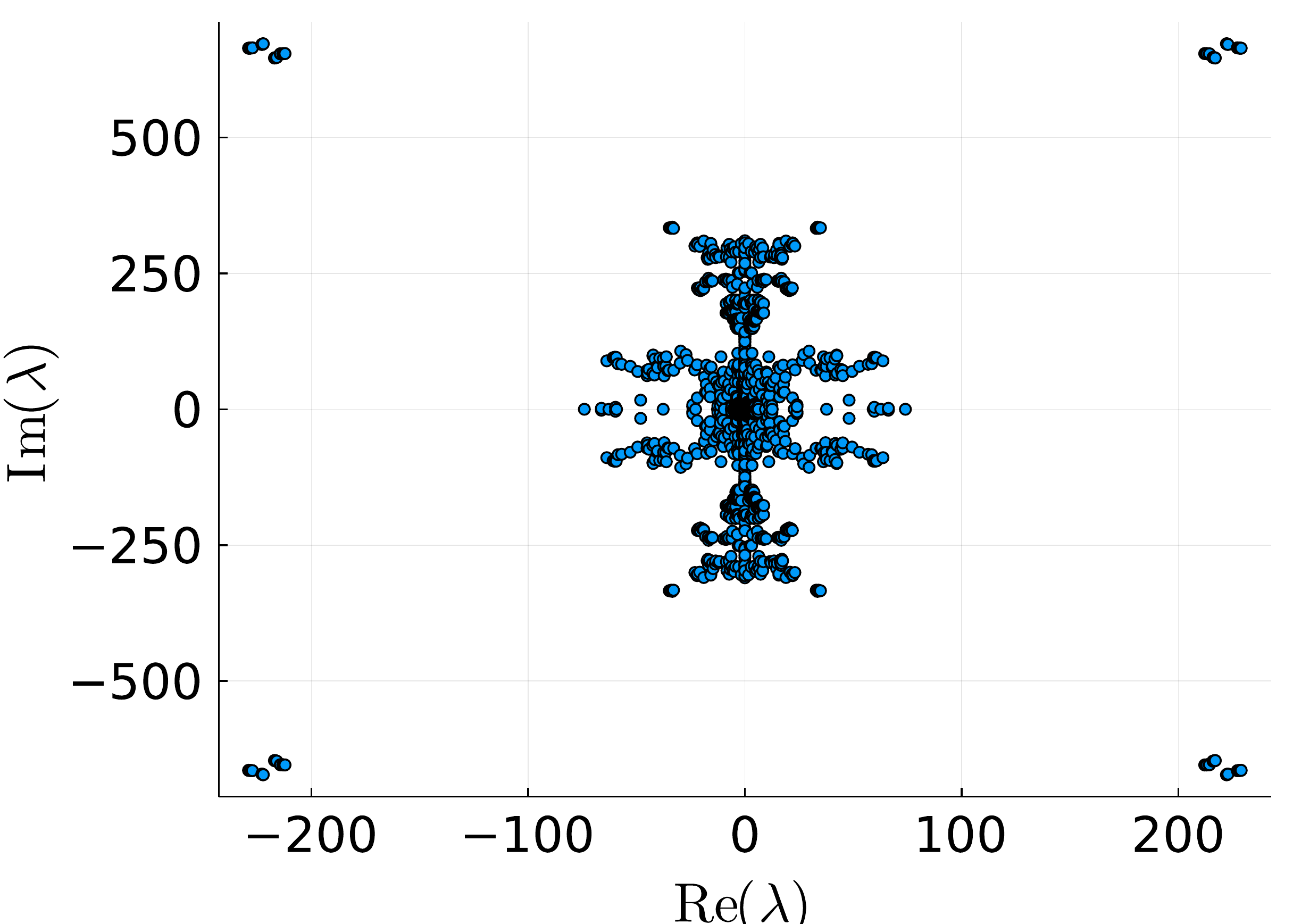}
    \caption{$\alpha = -2.2$, $\beta = 2$.}
  \end{subfigure}%
  \hspace*{\fill}
  \begin{subfigure}{0.33\textwidth}
  \centering
    \includegraphics[width=\linewidth]{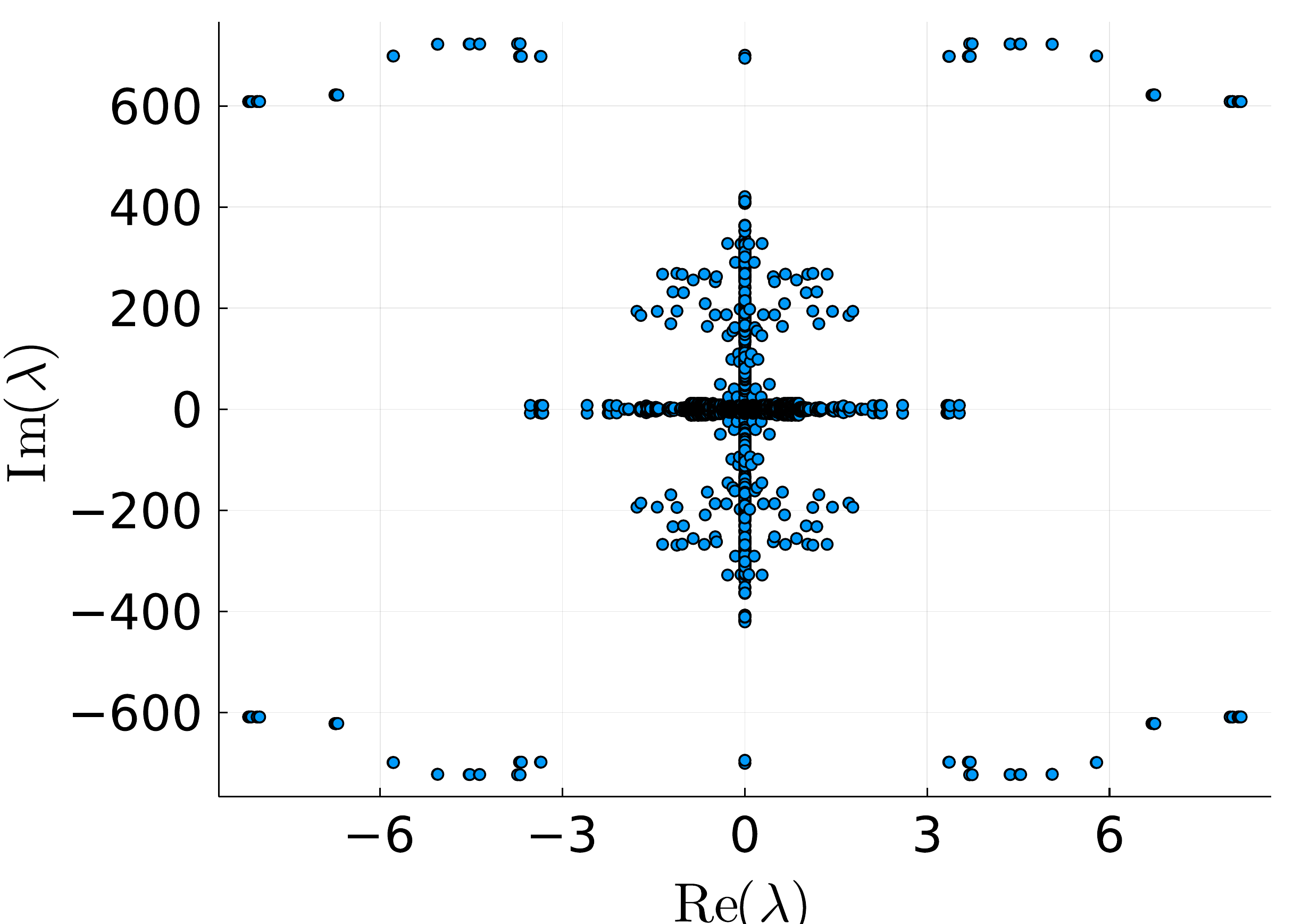}
    \caption{$\alpha = 1.0$, $\beta = -6$.}
  \end{subfigure}%
  \caption{Spectra of discontinuous Galerkin semidiscretizations of the 2D
           compressible Euler equations conserving Harten's entropy \eqref{eq:U} with parameter
           $\alpha$ using some values of $\beta = (\alpha + \gamma) / (1 - \gamma)$
           recommended for entropy splitting discretizations \cite{sjogreen2019entropy}
           and $\alpha = 1$.
           The EC discretizations use the flux of \cite{sjogreen2019entropy}
           and tensor product Lobatto Legendre bases with polynomials of degree
           five on $4 \times 4$ elements. Using the central flux instead results
           in a spectrum with negligible imaginary part
           \cite{gassner2022stability,ranocha2021preventing}.}
  \label{fig:spectra}
\end{figure}

Moreover, it is interesting to note that EC fluxes based on Harten's one-parameter
family of entropies
do not solve the local linear/energy stability issues discussed in
\cite{gassner2022stability,ranocha2021preventing}. This is demonstrated by
the spectra of discontinuous Galerkin semidiscretizations of the 2D compressible
Euler equations shown in Figure~\ref{fig:spectra}. The setup uses the density
wave example described in \cite{ranocha2021preventing}.
If the central flux $\vecfnum = \mean{\vec{f}}$ is used instead of an EC flux, the spectrum
of the resulting semidiscretization is essentially purely imaginary (ignoring
floating point errors etc.), indicating (marginal) local linear/energy stability of the method
\cite{gassner2022stability,ranocha2021preventing}. This property is desired
and reasonable for this physical setup, since the density wave reduces the
compressible Euler equations to linear advection equations, see also the extended
discussion in \cite{gassner2022stability,ranocha2021preventing}.

The semidiscretizations used to compute the spectra above are implemented using
Trixi.jl \cite{ranocha2022adaptive,schlottkelakemper2021purely}. The Jacobian
is computed using automatic differentiation \cite{revels2016forward} in Julia
\cite{bezanson2017julia}. All source code required to reproduce the examples
as well as additional material verifying the implementation and some calculations
presented in this article are available online \cite{ranocha2022noteRepro}.

\section*{Acknowledgments}

Funded by the Deutsche Forschungsgemeinschaft (DFG, German Research Foundation)
under Germany's Excellence Strategy EXC 2044-390685587, Mathematics M\"{u}nster:
Dynamics-Geometry-Structure.
Special thanks to Gregor Gassner and Ayaboe Edoh for discussions related to this manuscript.

\printbibliography

\end{document}